\documentclass[11pt]{article}

\usepackage{amsmath,amssymb,amsthm}

\usepackage[top=2.5cm, bottom=3cm, left=3.2cm, right=3.7cm]{geometry}

\usepackage[
  colorlinks,
  linkcolor=cyan,
  citecolor=black,
  urlcolor=black,
  bookmarks=true
]{hyperref}

\usepackage{xcolor}
\usepackage{bm}

\numberwithin{equation}{section}

\theoremstyle{definition}
\newtheorem{definition}{Definition}[section]
\newtheorem{proposition}[definition]{Proposition}
\newtheorem{theorem}[definition]{Theorem}
\newtheorem{corollary}[definition]{Corollary}

\newtheorem{lemma}[definition]{Lemma}

\newtheorem{conjecture}{Conjecture}
\newtheorem{maintheorem}{Theorem}
\newtheorem{claim}[definition]{Claim}

\newcommand{\QQ}{\mathbb{Q}}
\newcommand{\RR}{\mathbb{R}}
\newcommand{\CC}{\mathbb{C}}
\newcommand{\ZZ}{\mathbb{Z}}
\newcommand{\NN}{\mathbb{N}}

\newcommand{\B}{{\bm{\mathcal{B}}}}
\newcommand{\Per}{\mathrm{Per}}

\newcommand{\dlim}{\displaystyle\lim}
\newcommand{\dsum}{\displaystyle\sum}
\newcommand{\lnzd}[2]{\ell_{#1}\!\left(#2\right)}

\newcommand{\paper}[8]{
  \bibitem{#1} #2, ``\textit{#3}'', #4 \textbf{#5}{#6}, (#7), #8.
}
\newcommand{\book}[6]{
  \bibitem{#1} #2, ``\textit{#3}'', #4, #5, (#6).%
}

\allowdisplaybreaks

\title{Non-periodicity of the sequence of the last nonzero digits of factorials and its applications to transcendence}
\author{Kohta Gejima and Fumichika Takamizo}

\begin{document}
\maketitle

\renewcommand{\thefootnote}{\fnsymbol{footnote}}
\footnote[0]{\today} 
\footnote[0]{2020 Mathematics Subject Classification. Primary: 11J81\if0(Transcendence, general theory)\fi; Secondary: 11A63\if0(Radix representation; digital problems)\fi, 11B85\if0(Automatic sequences)\fi.}
\footnote[0]{Keywords: transcendence criterion, Cantor base expansion, last nonzero digit, automatic sequences.}
\renewcommand{\thefootnote}{\arabic{footnote}}

\begin{abstract}
  We prove that the sequence of the last nonzero digits of factorials in every integer base $b>2$ is not eventually periodic. We also extend the Adamczewski--Bugeaud criterion, originally formulated for integer base expansions, to Cantor base expansions associated with a periodic Cantor base. As an application, we show that a certain real number expressed through a Cantor base expansion is transcendental when the Cantor base and the digit sequence satisfy suitable conditions.
\end{abstract}

\tableofcontents

\section{Introduction}

The first explicit construction of transcendental numbers was given by Liouville~\cite{Liouville}. He considered certain real numbers that can be described in terms of their \textit{integer base expansions}, and proved that they are transcendental. The essential idea of his proof is that algebraic numbers cannot be approximated too well by rational numbers. Since then, integer base expansions have come to play a more prominent role in transcendental number theory.

In recent developments, Adamczewski and Bugeaud~\cite{AB} gave a powerful criterion for transcendence based on structural constraints on the digit sequence of an integer base expansion and settled the Cobham--Loxton--van der Poorten conjecture, which asserts that:

\begin{conjecture}
Irrational automatic numbers are transcendental.
\end{conjecture}

\noindent
Here an automatic number is that whose integer base expansion is produced by a finite automaton. A key notion in their approach is the \textit{stammering condition}, which captures the occurrence of partial repetitions of finite blocks within the digit sequence. Adamczewski and Bugeaud proved that numbers whose digit sequences satisfy the stammering condition are either rational or transcendental. Importantly, every automatic sequence satisfies the stammering condition. Thus their result not only settles the conjecture but also provides a concrete tool for constructing transcendental numbers from non-periodic automatic sequences.

Concrete examples of non-periodic automatic sequences include the Thue--Morse sequence (see \cite[Proposition~3.1.1]{Lothaire}), the paperfolding sequence~\cite{MFP}, and the Baum--Sweet sequence~\cite{Merta}. However, producing further concrete examples of non-periodic automatic sequences is difficult in general. As another candidate of a non-periodic automatic sequence, one may consider the sequence of the last nonzero digits of factorials in an integer base. Indeed, in base $10$, its non-periodicity was shown by Dresden~\cite{Dresden}, but, to the best of our knowledge, no reference is available for the case of general bases. In addition, a generalization of the criterion of Adamczewski and Bugeaud to more general \textit{Cantor base expansions} has not yet been considered, though a natural and interesting direction.

\medskip

The purposes of this paper are as follows.
\begin{itemize}
  \item[1.] To extend the Adamczewski--Bugeaud criterion from integer base expansions to Cantor base expansions.
  \item[2.] To prove that the sequence of the last nonzero digits of factorials for \textit{every} integer base $b>2$ is not eventually periodic. Note that in base $b=2$ the sequence is clearly periodic, so the case $b>2$ is of main interest.
\end{itemize}

Lipka~\cite{Lipka} obtained an exact characterization of the bases in which the sequence of the last nonzero digits of factorials is automatic. Consequently, these sequences satisfy the stammering condition. We apply our generalized criterion to construct explicit transcendental numbers via Cantor base expansions. Our extension makes it possible to derive a greater number of transcendental numbers from a single non-periodic automatic sequence. In particular, we have the following (Corollary \ref{corollary:lnzd-general}):

\begin{maintheorem}
Let $b>2$ be an integer with the prime factorization $b=p_1^{a_1}p_2^{a_2}\cdots$ such that $a_1(p_1-1) \geq a_2(p_2-1) \geq \cdots$. Let $\B=(\beta_k)_{k=1}^\infty$ be a real Cantor base of period $p$ consisting of real numbers $\beta_k>1$ such that $\delta=\prod_{k=1}^p\beta_k$ is Pisot. If $(b-1)s(\delta) \leq (\beta_{\mathrm{min}}-1)s(\beta_{\mathrm{min}}^p)$ and either $a_1(p_1-1)>a_2(p_2-1)$ or $b=p_1^{a_1}$, then the number
\begin{align*}
  \alpha_{b,\B} = \sum_{n=1}^\infty \lnzd{b}{n!}\prod_{k=1}^n \beta_k^{-1}
\end{align*}
is transcendental.
\end{maintheorem}
Here we denote by $\lnzd{b}{n!}$ the last nonzero digit of $n!$ in an integer base $b$ and put $\beta_{\mathrm{min}}=\min_i \beta_i$ and $s(x)=x/(x-1)$.

\medskip

This paper is organized as follows. In Section~\ref{sec:preriminaries}, we recall Cantor base representations, basic facts on automata and $k$-automatic sequences, and Evertse's Subspace Theorem. In Section~\ref{sec:TranscendenceCriterion}, we prove a transcendence criterion for Cantor base expansions associated with a periodic Cantor base. More precisely, after recalling stammering sequences and showing that every $k$-automatic sequence are stammering with explicit parameters, we prove a generalization of the Adamczewski-Bugeaud criterion (Theorem~\ref{theorem:main-nodiv}) to Cantor base representations. In Section~\ref{sec:nonperiodicity}, we prove the nonperiodicity of the sequence $(\lnzd{b}{n!})_{n=1}^\infty$ for every integer base $b>2$ (Theorem~\ref{theorem:MainResult2}) by using Legendre's formula and Gelfond's equidistribution theorem. Finally, combining the transcendence criterion and the nonperiodicity of $(\lnzd{b}{n!})_{n=1}^\infty$ for a fixed integer $b>2$, we obtain explicit transcendental numbers built from this sequence (Corollary~\ref{corollary:lnzd-general}).

\medskip

We use the following notation: $\ZZ$, $\NN$, $\QQ$, $\RR$ and $\CC$ denote the (rational) integers, the positive integers, the rational numbers, the real numbers and the complex numbers, respectively. For a real number $y$, we denote by $\lfloor y \rfloor$ and $\lceil y\rceil$ the greatest integer less than or equal to $y$ and the least integer greater than or equal to $y$, respectively, and write $\{y\}$ for the fractional part of $y$, that is, $\{y\}=y-\lfloor y\rfloor$.

\section{Preliminaries}\label{sec:preriminaries}

In this section, we begin by recalling the definitions of Cantor base representations in general form and a generalization of Schmidt's theorem. Next we recall notation for automata. Finally we state Evertse's Subspace Theorem, which is the key tool for proving of our criterion.

\subsection{Cantor base representations}

The Cantor base representation was rediscovered by Caalim and Demegillo \cite{CD} and independently by Charlier and Cisternino \cite{CC}. This is a representation of real numbers with respect to a Cantor base, which is a sequence of real bases. Its prototype already appeared in \cite{Cantor1} by Cantor.

Let $\B=(\beta_1,\beta_2,\cdots)$ be a sequence of complex numbers. We say that $\B$ is a Cantor base if $\B=(\beta_1,\beta_2,\cdots)$ satisfies ${\dlim_{n \rightarrow \infty}}\prod_{k=1}^n|\beta_k|=\infty$. If there exists $p \in \NN$ such that $\beta_{k+p}=\beta_k$ for any $k \geq 1$, we say that $\B$ is periodic, and $p$ is called a period of $\B$.

Let $\mathcal{A}$ be a finite subset of $\CC$. A $(\B,\mathcal{A})$-representation of $z \in \CC$ is an infinite word $\bm{a}=a_1a_2a_3\cdots$ over $\mathcal{A}$ such that
  \[
  z=\sum_{k=1}^\infty a_k \prod_{i=1}^k\beta_i^{-1}.
  \]
It is also called a Cantor base representation.

Let $\Gamma=(\Omega,\mathcal{A},G)$ be a $3$-tuple  consisting of two bounded subsets $\Omega=\Omega(\Gamma),\mathcal{A}=\mathcal{A}(\Gamma)$ of $\CC$ and a map $G=G(\Gamma):\bigcup_{i=1}^\infty \beta_i\Omega \rightarrow \mathcal{A}$ such that $\beta_i z-G(\beta_iz) \in \Omega$ for all $z \in \Omega$. For each $i \in \NN$, we define a map $T_i:\Omega \rightarrow \Omega$ by
\[
T_i(z)=\beta_i z-G(\beta_iz).
\]
Then an infinite word $\bm{a}=a_1a_2a_3\cdots$ is defined by
\[
  a_k=G(\beta_k \tau^{k-1}(z)) \quad ({}^\forall k=1,2,\cdots),
\]
where we put $\tau^{k-1}=T_{k-1}\circ \cdots \circ T_{1}$ and denote by $\tau^0$ the identity map on $\Omega$.

\begin{proposition}
Under the above assumptions, we have the expression of $z \in \Omega$ of the form
\begin{equation}
  z=\sum_{k=1}^\infty a_k \prod_{i=1}^k\beta_i^{-1}.
  \label{eq:(B,Gamma)-representation}
\end{equation}
\end{proposition}

\begin{proof}
  We can easily check that
  \[
    \tau^{n}(z)=\frac{a_{n+1}+\tau^{n+1}(z)}{\beta_{n+1}}
  \]
  for all nonnegative integers $n$. Thus we have
  \[
    z=\sum_{k=1}^n a_k \prod_{i=1}^k \beta_i^{-1}+\tau^n(z)\prod_{i=1}^n \beta_i^{-1}
  \]
  for all $n \in \NN$. Since $\Omega$ is bounded, we have
  \begin{align*}
    \left|z-\sum_{k=1}^n a_k \prod_{i=1}^k \beta_i^{-1}\right|
    =\left|\tau^n(z)\prod_{i=1}^n \beta_i^{-1}\right|
    =|\tau^n(z)|\prod_{i=1}^n |\beta_i|^{-1}
    \rightarrow 0\quad (n \rightarrow \infty).
  \end{align*}
\end{proof}

The infinite word $\bm{a}=a_1a_2a_3\cdots$ (or the expression (\ref{eq:(B,Gamma)-representation}) of $z$) is called the $(\B,\Gamma)$-expansion of $z$, which is represented by the symbol $d_{\B,\Gamma}(z)$. Put
\[
  \Per_{\Gamma}(\B)
  =\left\{z \in \Omega(\Gamma) \mid \text{$d_{\B,\Gamma}(z)$ is eventually periodic}\right\}.
\]

For $x \in [0,1)$, there exists a distinguished Cantor representation of $x$. Let $\B=(\beta_1,\beta_2,\cdots)$ be a Cantor base. Let $\Gamma=(\Omega,\mathcal{A},G)$ be the $3$-tuple with $\Omega=[0,1)$, $\mathcal{A}=\mathcal{A}_\B=\{0,1,\cdots,\sup_{n}\lfloor \beta_n\rfloor\}$, $G:\bigcup_{i=1}^\infty \beta_i\Omega \rightarrow \ZZ$, $y \mapsto \lfloor y \rfloor$. For $x \in [0,1)$, we denote by $d_\B(x)=d_{\B,\Gamma}(x)$ the ($\B$,$\Gamma$)-expansion of $x$. This is simply called the $\B$-expansion or the Cantor base expansion of $x$. In this case, we write $\Per(\B)$ for $\Per_\Gamma(\B)$. In particular, when $\B=(\beta,\beta,\beta,\cdots)$ with $\beta>1$, $d_\B(x)$ is identified with the classical $\beta$-expansion $d_\beta(x)$ of $x$.

\medskip

The following is a generalization of Schmidt's theorem \cite[Theorem~2.4 and 3.1]{Schmidt1}.

\begin{theorem}[\cite{CCK}, \cite{MP}]\label{theorem:Schmidt-general}
  Let $\B=(\beta_1,\beta_2,\beta_3,\cdots)$ be a Cantor base of period $p$ consisting of real numbers $\beta_k>1$. Put $\delta=\prod_{k=1}^p\beta_k$.
  \begin{enumerate}
    \item If $\QQ \cap [0,1) \subset \Per(\B)$, then $\delta$ is either a Pisot number or a Salem number and $\beta_1,\cdots,\beta_p \in \QQ(\delta)$.
    \item If $\delta$ is a Pisot number and $\beta_1,\cdots,\beta_p \in \QQ(\delta)$, then $\Per(\B) =\QQ(\delta) \cap [0,1)$.
  \end{enumerate}
\end{theorem}

\subsection{Automata}

An alphabet is a finite set of elements which are called letters. Let $\mathcal{A}$ be an alphabet. We denote by $\mathcal{A}^*$ the free monoid generated by $\mathcal{A}$. An element $W=a_1a_2a_3\cdots a_r$ of $\mathcal{A}^*$ is called a word of length $r$ on $\mathcal{A}$. The neutral element of $\mathcal{A}^*$ is called the empty word and its length is $0$. Let $\mathcal{A}^\NN$ be the set of all right infinite sequences on $\mathcal{A}$.

Let $k$ be a positive integer with $k>1$. A $k$-automaton is a $6$-tuple
\[
A=(Q,\Sigma_k,\delta,q_0,\Delta,\tau).
\]
Each element is as follows:
\begin{itemize}
  \item $Q$ is a finite set of symbols, called states.
  \item $\Sigma_k=\{0,1,\cdots,k-1\}$ is the set, called the input alphabet of $A$.
  \item $q_0$ is a fixed state, called the initial state.
  \item $\delta:Q\times \Sigma_k \rightarrow Q$ is a map, called the transition function.
  \item $\Delta$ is an alphabet, called the output alphabet.
  \item $\tau:Q \rightarrow \Delta$ is a map, called the output function.
\end{itemize}
  We always assume that $\delta(q_0,0)=q_0$. We extend $\delta$ to $\delta:Q \times \Sigma_k^* \rightarrow Q$ so that
  \[
    \delta(q,W W')=\delta(\delta(q,W),W')\qquad ({}^\forall q \in Q,\ {}^\forall W,W' \in \Sigma_k^*).
  \]

We say that an infinite word $W=a_1 a_2 a_3 \cdots$ is $k$-automatic if there exists a $k$-automaton $A=(Q,\Sigma_k,\delta,q_0,\Delta,\tau)$ such $a_n=\tau(\delta(q_0,(n)_k))$, where $(n)_k$ is the $k$-ary representation of $n$.

\subsection{Evertse's Subspace Theorem}

Let $K$ be an algebraic number field of degree $d=[K:\QQ]$. Our proof of the main result is based on the Subspace Theorem by Evertse \cite{Evertse}. In this section, we recall the definition of the height $H(\bm{x})$ for $\bm{x} \in K^n$.

We denote by $M_K$ the set of places on $K$ and by $S_\infty(K)$ the set of infinite places on $K$. For each infinite place $v \in S_\infty(K)$, $e_v$ denotes $1$ or $2$ according as $v$ is real or not. Then we have $d=\sum_{v \in S_\infty(K)}e_v$. Let $|\cdot|$ be the usual absolute value on $\CC$, that is, $|x|=(x\overline{x})^{1/2}$ for $x \in \CC$. For any place $v \in M_K$, an absolute value $|x|_v$ of $x \in K$ is chosen as follows:
\begin{itemize}
  \item If $v$ is a finite place corresponding to the prime ideal $\mathfrak{p}$ of $O_K$, then $|x|_v=|O_K/\mathfrak{p}|^{-\mathrm{ord}_\mathfrak{p}(x)}$, where $O_K$ is the ring of integers of $K$.
  \item If $v$ is a real infinite place corresponding to an embedding $\sigma: K \hookrightarrow \RR$, then $|x|_v=|x^\sigma|^{e_v/d}=|x^\sigma|^{1/d}$.
  \item If $v$ is a non-real infinite place corresponding to the conjugate pair $\{\sigma,\overline{\sigma}\}$ for a non-real complex embedding $\sigma: K \hookrightarrow \CC$, then $|x|_v=|x^\sigma|^{e_v/d}=|x^\sigma|^{2/d}$.
\end{itemize}

These absolute values satisfy the product formula for $K$:
\[
  \prod_{v \in M_K}|x|_v=1\quad (x \in K^\times).
\]

Fix an infinite place $v$ corresponding to the embedding $\sigma: K \hookrightarrow \CC$. We denote by the same symbol $|\cdot|_v$ its extension to $\CC$ given by
\[
|z|_v:= |z|^{e_v/d}\qquad (z \in \CC),
\]
In particular, for $x\in K$, we have
\[
|x|_v = |\sigma(x)|^{e_v/d}.
\]

Let $\bm{x}=(x_1,\cdots,x_n) \in K^n$. For each $v \in M_K$, put
\begin{align*}
  |\bm{x}|_v&=\left( \sum_{k=1}^n|x_k|_v^{2d/e_v}\right)^{e_v/(2d)}
  \quad (\text{if $v$ is infinite});\\
  |\bm{x}|_v&=\max \left\{|x_k|_v\bigm| k=1,\cdots,n\right\}
  \quad (\text{if $v$ is finite}).
\end{align*}
Then we define the height of $\bm{x}$ by
\[
  H(\bm{x})=H(x_1,\cdots,x_n)
  =\prod_{v \in M_K}|\bm{x}|_v.
\]
Note that $H(\bm{x})$ is independent of the choice of the number field $K$ containing the coordinates of $\bm{x}$ (see \cite{Evertse}, for example).

The Subspace Theorem by Evertse is as follows.

\begin{theorem}[Evertse \cite{Evertse}]\label{theorem:SubspaceTheorem}
  Let $K$ be an algebraic number field. Let $m \geq 2$ be an integer. Let $S$ be a finite set of places on $K$ containing all infinite places. For each $v \in S$, let $L_{1,v},\cdots,L_{m,v}:K^m \rightarrow \CC$ be linear forms with algebraic coefficients and with
  \[
    \mathrm{rank}\left\{L_{1,v},\cdots,L_{m,v}\right\}=m.
  \]
  Let $\varepsilon$ be a real number with $0<\varepsilon<1$. Then the set of solutions $\bm{x} \in K^m$ to the inequality
  \[
    \prod_{v \in S}\prod_{i=1}^m\frac{|L_{i,v}(\bm{x})|_v}{|\bm{x}|_v}\leq H(\bm{x})^{-m-\varepsilon}
  \]
  lies in finitely many proper subspaces of $K^m$.
\end{theorem}

\section{The Adamczewski--Bugeaud criterion}\label{sec:TranscendenceCriterion}

In this section, we prove our first main result, which is the Adamczewski--Bugeaud criterion for Cantor base representations. We begin by recalling the definition of stammering sequences, introduced by Adamczewski, Bugeaud and Luca \cite{ABL} and prove that all automatic sequences are stammering. Next we prove the criterion in five steps.

\subsection{Stammering sequences}\label{ss:StammeringSequences}

Let $W$ be a word. We denote by $|W|$ the length of $W$. For any $\ell \in \NN$, we write $W^\ell$ for the $\ell$ repetitions of the word $W$. More generally, for any real number $x>0$, we define $W^x$ to be the word $W^{[x]}W'$, where $W'$ is the prefix of $W$ of length $\lceil \{x\}|W|\rceil=\lceil (x-[x])|W|\rceil$.

Let $\bm{a}=(a_n)_{n=1}^\infty$ be a sequence of elements of $\mathcal{A}$. We sometimes identify it with the infinite word $a_1a_2a_3\cdots$. Let $\omega>1$ and $M>0$ be real numbers. If there exist two sequences $(U_n)_{n=1}^\infty$ and $(V_n)_{n=1}^\infty$ of finite words satisfying the following three conditions, then we say that $\bm{a}$ is $(\omega,M)$-stammering (witnessed by $(U_n)_{n=1}^\infty$, $(V_n)_{n=1}^\infty$), or simply ``stammering'':
\begin{itemize}
\item For any $n \geq 1$, the word $U_n V_n^\omega$ is a prefix of the word $\bm{a}$.
\item For $n$ sufficiently large, $|U_n|/|V_n| \leq M$ holds.
\item $\dlim_{n \rightarrow \infty}|V_n|=\infty$.
\end{itemize}

\medskip

Adamczewski and Bugeaud showed that a morphic sequence generated by a $k$-uniform morphism is stammering in their proof of Theorem~2 in \cite{AB}. For the convenience of the reader, we give the proof in the terms of $k$-automatic sequences.

\begin{proposition}\label{proposition:stam_if_k-auto}
If a sequence $\bm{a}=(a_n)_{n=1}^\infty$ is $k$-automatic, then $\bm{a}$ is stammering. More precisely, if $\bm{a}$ is produced by a $k$-automaton
$A=(Q,\Sigma_k,\delta,q_0,\Delta,\tau)$ with $\delta(q_0,0)=q_0$, then $\bm{a}$ is $(\omega,M)$-stammering with
\[
\omega = 1+\frac{1}{|Q|},\qquad M =|Q|-1,
\]
where $|Q|$ is the cardinality of $Q$.
\end{proposition}

\begin{proof}
Let $A=(Q,\Sigma_k,\delta,q_0,\Delta,\tau)$ be as in the statement, and write
\[
a_n=\tau(\delta(q_0,(n)_k))\qquad(n \geq 1).
\]
Recall that $(n)_k$ denotes the $k$-ary expansion of $n$.

\medskip
\noindent
\textit{Step 1.}
For $m \geq 1$ and $q \in Q$, define a word $B_m(q)$ on $\Delta$ of length $k^m$ by
\[
  (B_m(q))[r]=\tau(\delta(q,(r-1)_k^{(m)})) \quad (1 \leq r \leq k^m),
\]
where $(r-1)_k^{(m)}$ is the word of length $m$ which is obtained by padding the $k$-ary expansion of $r-1$ with leading zeros so as to make its length $m$. For each $i \geq 0$, put $q_i=\delta(q_0,(i)_k)\in Q$.

\begin{claim}\label{claim:B_m}
For all $i \geq 0$ and all $m \geq 1$, we have
\[
  a_{ik^m+r}=B_m(q_i)[r] \qquad (1 \leq r \leq k^m).
\]
\end{claim}

\begin{proof}[Proof of Claim~\ref{claim:B_m}]
For $1 \leq r \leq k^{m}$, we consider the $k$-ary expansion of $ik^m+r-1$. Since the last $m$ digits of $ik^m$ are all zero, we have $(ik^m+r-1)_k=(i)_k (r-1)_k^{(m)}$, where the right-hand side is the concatenation of $(i)_k$ and $(r-1)_k^{(m)}$. Hence, by the inductive definition of $\delta$, we have
\begin{align*}
\delta(q_0,(ik^m+r-1)_k)
&=\delta(q_0,(i)_k (r-1)_k^{(m)})\\
&=\delta(\delta(q_0,(i)_k),(r-1)_k^{(m)})
=\delta(q_i,(r-1)_k^{(m)}).
\end{align*}
By applying $\tau$, we have $a_{ik^m+r}=(B_m(q_i))[r]$.
\end{proof}

For each fixed $m$, the sequence $\bm{a}$ decomposes blockwise as
\[
\bm{a} = C_0C_1C_2\cdots,\qquad
C_i=B_m(q_i).
\]
Each $C_i$ depends only on the state $q_i$. Hence there exist at most $|Q|$ distinct blocks among $(C_i)_{i=0}^\infty$.

\medskip
\noindent\textit{Step 2.}
Consider the first $|Q|+1$ blocks $C_0,C_1,\dots,C_{|Q|}$. By the pigeonhole principle,
there exist $r$ and $s$ with $0\leq r < s \leq |Q|$ such that $C_r=C_s$. Put $t=s-r\geq 1$ and define
\[
U_m = C_0 C_1 \cdots C_{r-1},\qquad
V_m = C_r C_{r+1}\cdots C_{s-1}.
\]
Then $\bm{a}$ begins with $U_mV_m$ and the next block is $C_s=C_r$, that is, the first $1/t$-portion of $V_m$ immediately appears. We put
\[
  \omega=1+\frac{1}{|Q|}.
\]
Then
\[
  \omega=1+\frac{1}{|Q|} \leq 1+\frac{1}{t}.
\]
Thus we conclude that $U_mV_m^{\omega}$ is a prefix of $\bm{a}$.

\medskip
\noindent\textit{Step 3.}
Recall that each block $C_i$ has length $k^m$. Thus
\[
  |U_m|=r k^m,\qquad |V_m|=t k^m.
\]
Therefore
\[
  \frac{|U_m|}{|V_m|} = \frac{r}{t} \leq \frac{s-1}{t} \leq |Q|-1,
\]
so we may take $M=|Q|-1$. Moreover, $|V_m|=t k^m \geq k^m$ is strictly increasing in $m$.

\smallskip
Altogether, for each $m \geq 1$, the word $U_mV_m^{\omega}$ is a prefix of $\bm{a}$, the inequality $|U_m|/|V_m| \leq M$ holds, and the block sequence $(|V_m|)_{m=1}^\infty$ is strictly increasing. Hence $\bm{a}$ is $(\omega,M)$-stammering with the parameters $\omega = 1+1/|Q|$ and $M=|Q|-1$.
\end{proof}

\subsection{The Adamczewski--Bugeaud criterion for Cantor base representations}\label{ss:MainResults}

In this subsection, we prove a transcendental criterion for $(\B,\mathcal{A})$-representations in five steps.

Let $K$ be an algebraic number field and $S$ a set of places on $K$. We define $F_S:K \rightarrow \RR$ by
\[
  F_S(x)=\sum_{v \in S}\log f_v(x)\quad (x \in K).
\]
Here, for each place $v$, $f_v:K \rightarrow \RR$ is defined by $f_v(x)=\max(1,|x|_v)$ for all $x \in K$.

We denote by $S_\infty^{(1)}(\delta;K)$ (resp. $S_\infty^{(2)}(\delta;K)$) the set of infinite places $v$ on $K$ satisfying $|\delta|_v>1$ (resp. $|\delta|_v \leq 1$). Then $S_\infty(K)=S_\infty^{(1)}(\delta;K) \sqcup S_\infty^{(2)}(\delta;K)$.

Our first main result is as follows.

\begin{theorem}\label{theorem:main-nodiv}
Let $p$ be a positive integer, $\beta_1,\cdots,\beta_p$ algebraic numbers, $\B=(\beta_i)_{i=1}^\infty$ the Cantor base of period $p$. Put $\delta=\prod_{i=1}^p\beta_i$, $K=\QQ(\beta_1,\cdots,\beta_p)$ and $S=S_\infty(K)$. Let $w$ be the infinite place corresponding to the identity embedding $K \hookrightarrow \CC$. Let $\mathcal{A}$ be a finite subset of $K$ and $\bm{a}=(a_k)_{k=1}^\infty$ an $(\omega,M)$-stammering sequence on $\mathcal{A}$. If
\[
  \frac{\omega+M}{1+M} > \frac{F_S(\delta)}{\log|\delta|_w},
\]
then
\[
\alpha=\sum_{k=1}^\infty a_k\prod_{i=1}^k\beta_i^{-1}
\]
either belongs to $K$ or is transcendental.
\end{theorem}

We write $A_n \lesssim B_n$ to mean that there exists a constant $C>0$ independent of $n$ such that $A_n \leq C\,B_n$.
All implied constants may depend on the fixed data $(\B,\mathcal{A},K)$ but never on $n$.

For each $k \geq 1$, we define $\theta_k$ by
\[
  \theta_k=\prod_{i=1}^{r}\beta_i \qquad(\theta_k=1 \text{ if } r=0),
\]
where we write $k=pt+r$ with $0 \leq r<p$. Then, for every place $v\in M_K$, there exists a constant $C_v \geq 1$ such that
\begin{equation}
  C_v^{-1}\ \leq |\theta_k|_v \leq C_v,\label{eq:bounded_tails}
\end{equation}
and $\alpha$ can be expressed as
\[
  \alpha=\sum_{k=1}^\infty a_k \theta_k^{-1} \delta^{-\lfloor k/p \rfloor},\qquad
  \theta_k \in \{\theta_0,\cdots,\theta_{p-1}\}.
\]

\begin{proof}[Proof of Theorem~\ref{theorem:main-nodiv}]
Let $\bm{a}$ be $(\omega,M)$-stammering witnessed by $(U_n)_{n=1}^\infty$, $(V_n)_{n=1}^\infty$ and we put $u_n=|U_n|$ and $v_n=|V_n|$. Put $\psi_n=u_n+\lceil \omega v_n\rceil$. Let $\bm x_n=(x_{n,1},x_{n,2},x_{n,3})\in K^3$, where
\[
  x_{n,1}=\delta^{\lfloor (u_n+v_n)/p\rfloor},\qquad
  x_{n,2}=\delta^{\lfloor u_n/p\rfloor},\qquad
  x_{n,3}=-\,Q_n.
\]
Here we put
\[
  Q_n=\bigl(x_{n,1}-x_{n,2}\bigr)\,\sum_{k=1}^{\psi_n-1} a_k \theta_k^{-1} \delta^{-\lfloor k/p\rfloor}.
\]

For each $v\in S$, we define linear forms $L_{i,v}:K^3 \rightarrow \CC$ by
\[
  L_{1,v}(x_1,x_2,x_3)=x_1,\quad
  L_{2,v}(x_1,x_2,x_3)=x_2,\quad
  L_{3,v}(x_1,x_2,x_3)=
  \begin{cases}
    \alpha x_1-\alpha x_2+x_3&(v=w),\\
    x_3&(v\neq w).
  \end{cases}
\]

\medskip
\noindent\textit{Step 1.}
By construction,
\[
  L_{3,w}(\bm x_n)
  =\left(x_{n,1}-x_{n,2}\right)\sum_{k=\psi_n}^{\infty} a_k \theta_k^{-1} \delta^{-\lfloor k/p\rfloor}.
\]
Since $|\delta|_w>1$ and $\lfloor (u_n+v_n)/p \rfloor \geq \lfloor u_n/p \rfloor$, we have
\[
  |x_{n,1}-x_{n,2}|_w \leq |x_{n,1}|_w+|x_{n,2}|_w \leq 2|x_{n,1}|_w \lesssim |x_{n,1}|_w.
\]
Hence, by (\ref{eq:bounded_tails}), we have
\begin{align*}
  |L_{3,w}(\bm x_n)|_w
  &\leq |x_{n,1}-x_{n,2}|_w\sum_{k=\psi_n}^{\infty} |a_k|_w\,|\theta_k^{-1}|_w\,|\delta|_w^{-\lfloor k/p\rfloor}\\
  &\lesssim  |x_{n,1}|_w \sum_{k=\psi_n}^{\infty} |\delta|_w^{-k/p}
  \lesssim\ |\delta|_w^{(u_n+v_n)/p}\cdot
                |\delta|_w^{-\psi_n/p}
  \lesssim |\delta|_w^{-(\omega-1) v_n/p}.
\end{align*}

\medskip
\noindent\textit{Step 2.}
If $v\in S_\infty^{(1)}(\delta;K)$, then, by (\ref{eq:bounded_tails}),
\[
  |x_{n,1}|_v\lesssim |\delta|_v^{(u_n+v_n)/p},\quad
  |x_{n,2}|_v\lesssim |\delta|_v^{u_n/p},\quad
  |x_{n,3}|_v\lesssim |\delta|_v^{(u_n+v_n)/p}.
\]
Let $v\in S_\infty^{(2)}(\delta;K)$. Then
\[
  |x_{n,1}|_v,\,|x_{n,2}|_v\lesssim 1.
\]
Moreover, since $|\delta^{-\lfloor k/p \rfloor}|_v \leq |\delta^{-1/p}|_v$ for all $k=1,\cdots\psi_n-1$, we have
\[
  |x_{n,3}|_v \lesssim\ v_n,
\]
by using $u_n/v_n\leq M$ for sufficiently large $n$ and $\dlim_{n \rightarrow \infty}v_n=\infty$.

\medskip
\noindent\textit{Step 3.}
Let $E=\sum_{v\in S_\infty^{(2)}(\delta;K)} e_v$.
By Step~2, for $v\in S=S_\infty(K)$,
\[
  |\bm{x}_n|_v\ \lesssim\
  \begin{cases}
    |\delta|_v^{(u_n+v_n)/p}&(v\in S_\infty^{(1)}(\delta;K)),\\[2pt]
    v_n^{e_v/d}&(v\in S_\infty^{(2)}(\delta;K)).
  \end{cases}
\]
Therefore the $S$-height $H_S(\bm{x}_n):=\prod_{v\in S}|\bm{x}_n|_v$ is estimated for sufficiently large $n$ as follows:
\begin{align*}
  H_S(\bm{x}_n)
  &\lesssim \left(\prod_{v\in S_\infty^{(1)}(\delta;K)} |\delta|_v^{(u_n+v_n)/p}\right) v_n^{E/d}
  =|\delta|_w^{\frac{(u_n+v_n)}{p}\cdot\frac{F_S(\delta)}{\log|\delta|_w}}\cdot v_n^{E/d}
  \lesssim\ |\delta|_w^{\frac{v_n(1+M)}{p}\cdot\frac{F_S(\delta)}{\log|\delta|_w}}\cdot v_n^{E/d}.
\end{align*}

\medskip
\noindent\textit{Step 4.}
Put
\[
  \Pi_n=\prod_{v\in S}\prod_{i=1}^3 \frac{|L_{i,v}(\bm{x}_n)|_v}{|\bm{x}_n|_v}.
\]
Since $|x_{n,j}|_v\le |\bm{x}_n|_v$ for $j=1,2$ and every $v\in S$, we have
\[
  \prod_{v\in S}\frac{|x_{n,j}|_v}{|\bm{x}_n|_v}\le 1\qquad(j=1,2).
\]
Hence we have
\[
  \Pi_n\ \le\ \prod_{v\in S}\frac{|L_{3,v}(\bm x_n)|_v}{|\bm{x}_n|_v^3}.
\]
By using Steps~1--3 and the definition of $F_S(\delta)$, for sufficiently large $n$, we obtain
\begin{align*}
  \Pi_n
  &\lesssim
  |\delta|_w^{-\frac{(\omega-1)}{p}v_n}\,
  \left(\prod_{v\in S_\infty^{(1)}(\delta;K)\setminus\{w\}}
          |\delta|_v^{\frac{(u_n+v_n)}{p}}\right)\,
  v_n^{E/d}\cdot H_S(\bm{x}_n)^{-3}\\
  &= v_n^{E/d}\,
    |\delta|_w^{\frac{(u_n+v_n)}{p}\left(\frac{F_S(\delta)}{\log|\delta|_w}-1\right)
    -\frac{(\omega-1)}{p}v_n}\cdot H_S(\bm{x}_n)^{-3}\\
  &\le v_n^{E/d}\,
    |\delta|_w^{-\frac{v_n}{p}\left\{\omega+M-(1+M)\frac{F_S(\delta)}{\log|\delta|_w}\right\}}
    \cdot H_S(\bm{x}_n)^{-3}.
\end{align*}
By the hypothesis
$\frac{F_S(\delta)}{\log|\delta|_w} \leq \frac{\omega+M}{1+M}$, there exists $\varepsilon>0$ such that
\[
  \Pi_n\,H_S(\bm x_n)^{3+\varepsilon}
  \ \lesssim\ v_n^{(E+\varepsilon)/d}\,
  |\delta|_w^{-\frac{v_n}{p}\left\{\omega+M-(1+M)\frac{F_S(\delta)}{\log|\delta|_w}
  -\varepsilon(1+M)\frac{F_S(\delta)}{\log|\delta|_w}\right\}}
  \rightarrow 0.
\]
Therefore we obtain $\Pi_n \le H_S(\bm{x}_n)^{-3-\varepsilon}$ for all sufficiently large $n$.

\medskip
\noindent\textit{Step 5.}
We assume that $\alpha$ is algebraic. By Evertse's Subspace Theorem (with $S=S_\infty(K)$), the points $\bm{x}_n$ lie in finitely many proper subspaces of $K^3$. Hence, by the pigeonhole principle, there exists $(X,Y,Z)\in K^3\setminus\{(0,0,0)\}$ with
\[
  X\,x_{n,1}+Y\,x_{n,2}+Z\,x_{n,3}=0
\]
for infinitely many $n$. Since $\alpha x_{n,1}-\alpha x_{n,2}-Q_n=L_{3,w}(\bm{x}_n)$, we have
\[
  (X-Z\alpha)\,x_{n,1}+(Y+Z\alpha)\,x_{n,2}=-Z\,L_{3,w}(\bm{x}_n).
\]
By Step~1, the limit of the right-hand side along an appropriate subsequence of positive integers $n$ tends to $0$. On the other hand, since $|x_{n,1}|_w \asymp |\delta|_w^{(u_n+v_n)/p}$ and $|x_{n,2}|_w\asymp |\delta|_w^{u_n/p}$ have distinct growth, the left-hand side converges to $0$ if and only if $X-Z\alpha=0$ and $Y+Z\alpha=0$. In particular, we have $\alpha=X/Z\in K$. This proves the theorem.
\end{proof}

\begin{corollary}\label{corollary:Pisot-simplifies}
If $|\delta|_w>1$ and $|\delta|_v\leq 1$ for all infinite places $v \not= w$ (for example, $\delta$ is a Pisot number), then
\[
\alpha=\sum_{k=1}^\infty a_k\prod_{i=1}^k\beta_i^{-1},
\]
where $\bm{a}=(a_k)_{k=1}^\infty$ is $(\omega,M)$-stammering with $\omega>1$, either belongs to $K$ or is transcendental.
\end{corollary}

\begin{proof}
  In this case, we have $F_S(\delta)=\log|\delta|_w$. Thus the hypothesis of Theorem~\ref{theorem:main-nodiv} reduces to $\omega>1$.
\end{proof}

\section{Non-periodicity of the sequence of the last nonzero digits of factorials}\label{sec:nonperiodicity}

In this section, we prove our second main result, which is the non-periodicity of the sequence of the last nonzero digits of factorials. First we recall the definition of the last nonzero digits, Legendre's formula and Gelfond's equidistribution theorem. Next we prove the non-periodicity in three steps. Finally we give examples of transcendental numbers as an application of our main results.

\subsection{Preliminaries}

Let $b\geq2$ and $n\geq1$ be positive integers. We put
\[
U_b(n)=\frac{n}{b^{v_b(n)}}.
\]
We define $\lnzd{b}{n}$, which is called the last nonzero digit of $n$ in base $b$, by two properties
\[
  \lnzd{b}{n}\equiv U_b(n) \pmod b,
  \qquad 1\leq \lnzd{b}{n}\leq b-1,
\]
where $v_b(n)=\max\{e \geq 0 \mid\  b^e\mid n\}$. We write $s_b(n)$ for the sum of base-$b$ digits of $n$.

Then our second main result is as follows.

\begin{theorem}\label{theorem:MainResult2}
  Let $b \geq 2$. Then the sequence $(\lnzd{b}{n!})_{n=1}^\infty$ is eventually periodic if and only if $b=2$.
\end{theorem}
This theorem can be expressed in terms of the ``unit part'' as follows.

\begin{corollary}
  Let $b>2$. Then, for all $\lambda \geq 1$ and all $N \geq 1$, there exists $n \geq N$ such that
  \[
    U_b((n+\lambda)!) \not\equiv U_b(n!) \pmod{b}.
  \]
\end{corollary}

In the case where $b=2$, it is trivial that the sequence $(\lnzd{b}{n!})_{n=1}^\infty$ is periodic. Hence it is enough to consider the case where $b>2$.

\medskip

Our fundamental tools to prove Theorem~\ref{theorem:MainResult2} are the following.

\begin{proposition}[Legendre's formula]\label{proposition:Legendre}
For any prime number $p$ and $n \in \NN$, we have
\begin{equation*}
v_{p}(n!) = \sum_{i=1}^\infty \left\lfloor \frac{n}{p^i} \right\rfloor
= \frac{n-s_{p}^{}(n)}{p-1}
.
\end{equation*}
\end{proposition}

\begin{theorem}[Gelfond \cite{Gelfond}]\label{theorem:Gelfond}
Let $b \geq 2$ and $k \geq 2$ be integers such that $\gcd(k,b-1)=1$. Let $m \geq 2$, and $a$ and $r$ be integers. Then, for some $\delta=\delta(b,k,m)$ with $0<\delta<1$, we have
\[
\#\left\{n < x \bigm| s_b(n) \equiv a \pmod{k},\ n \equiv r \pmod{m} \right\}
=\frac{x}{km}+O\left(x^{\delta}\right).
\]
\end{theorem}

It is easy to see the following three lemmas.

\begin{lemma}\label{lemma:multiplicative}
Let $b\geq2$. If $b \nmid \lnzd{b}{M}\,\lnzd{b}{N}$, then
\[
  \lnzd{b}{MN} \equiv \lnzd{b}{M}\cdot \lnzd{b}{N} \pmod b.
\]
\end{lemma}

\begin{lemma}\label{lemma:conservation}
Let $q \geq 2$, $0 \leq v_q(x) < A$ and $u \in \ZZ$. Then
\[
  v_q\left(x+q^A u \right) = v_q(x).
\]
\end{lemma}

\begin{lemma}\label{lemma:nocarry}
Let $q \geq 2$, $A \geq 1$, $0 \leq x < q^A$ and $m \geq 0$. Then
\[
  s_q\left(x+q^A m \right) = s_q(x) + s_q(m).
\]
\end{lemma}

We also use the following bound: if $p^M \leq n < p^{M+1}$ then
\begin{equation}\label{eq:s_bound}
s_{p}(n)\leq (p-1)(M+1)\leq (p-1)(\log_p n+1).
\end{equation}

\subsection{The proof of the non-periodicity}

In this subsection, we prove Theorem~\ref{theorem:MainResult2}. The proof consists of three steps. It is enough to show the following.

\begin{theorem}\label{theorem:MainResult2Sub}
If $b>2$, then the sequence $\left(\lnzd{b}{n!}\right)_{n=1}^\infty$ is not eventually periodic.
\end{theorem}

\begin{proof}
Assume for a contradiction that the sequence is eventually periodic, that is, there exist $N_0$ and $\lambda \geq 1$ such that
\begin{equation}\label{eq:periodicity}
  \lnzd{b}{(n+\lambda)!}=\lnzd{b}{n!}\qquad({}^\forall n\geq N_0).
\end{equation}

Write the prime factorization of $b$ as $b=\displaystyle\prod_{i=1}^r p_i^{a_i}$ and put $\Theta_{i}=a_{i}(p_i-1)$ and
\[
I_*=\left\{ \iota \in\{1,\cdots,r\} \biggm| \Theta_{\iota}=\max_{1\leq j \leq r}\Theta_{j}\right\}.
\]
We choose $j_* \in I_*$ so that $a_{j_*} \log p_{j_*}=\displaystyle\min_{j \in I_*} a_{j} \log p_j$ and put $q=p_{j_*}$, $a=a_{j_*}$ and $\Xi_q=\Theta_{j_*}$. Then
\[
  \Xi_q=\max_{1\leq j \leq r}\Theta_{j},\qquad a \log q=\min_{j \in I_*} a_{j} \log p_j.
\]
By the uniqueness of prime factorization, we have
\[
a\log q < a_j \log p_j
\]
for all $j \in I_* \setminus \{j_*\}$.

Write
\[
  b=q^a K,\qquad \gcd(q,K)=1.
\]
Then $a \geq 1$ if $q$ is odd and otherwise $a \geq 2$. Since any multiple of a period is also a period, we may assume that
\[
  \lambda=b^{\,i}L=q^{ai}K^{i}L \quad (i\geq 0),\qquad \gcd(L,b)=1.
\]
If necessary, we may assume that $i$ is as large as needed in what follows.

\medskip

\noindent
\textit{Step 1.}
We decompose $a=q^{v_{q}(a)}c$ with $\gcd(c,q)=1$, and moreover $c=c_1c_2$ with $c_1 \mid (q-1)$ and $\gcd(c_2,q-1)=1$. Put
\[
  \mu_0=\mathrm{lcm}\bigl(a, \mathrm{ord}(q;\ZZ/c\ZZ)\bigr),\qquad
  \mu=(i+1)\mu_0.
\]
Here we put
\[
  \mathrm{ord}(q;\left(\ZZ/c\ZZ\right)^\times)=\min\{t \mid q^t \equiv 1 \pmod{c}\}.
\]
Then $\mu \equiv0 \pmod{a}$, $q^\mu \equiv 1 \pmod{c}$, and $\mu-ai \geq a$.

Put
\[
n_0=q^\mu(1+u)-1
\]
for an integer $u$ to be conveniently chosen later. To choose a suitable $u$, we use Theorem~\ref{theorem:Gelfond} in base $q$ as follows: Put
\[
k=q^{v_{q}(a)}c_2,\qquad
m=\mathrm{lcm}(c_2,K^{i}).
\]
Since $q$ was chosen so that $a\log q< a_j \log p_j$ holds, we can choose $\alpha > 0$ so that $a_j\log p_j>(\alpha+1)a\log q$ for all $j \in I_* \setminus \{j_*\}$. Fix such an $\alpha \in (0,1]$. By Theorem~\ref{theorem:Gelfond}, for all large $\mu$, there exists $u \leq q^{\alpha \mu}$ such that
\begin{align*}
  \begin{cases}
    s_q(u) \equiv -1 \pmod{k},\\
    u\equiv -1 \pmod{m}.
  \end{cases}
\end{align*}

By Lemma~\ref{lemma:vq-mod-a} described later, we have $v_q(n_0!)\equiv 0 \pmod{a}$.

\smallskip
\begin{claim}\label{claim}
For all sufficiently large $i$ (equivalently large $\mu$), we have
\[
  v_b(n_0!)=\min_{1\leq j\leq r}\left\lfloor \frac{v_{p_j}(n_0!)}{a_j}\right\rfloor
  = \frac{v_q(n_0!)}{a}.
\]
\end{claim}

\begin{proof}[Proof of Claim~\ref{claim}]
If $r=1$, no further work is required. Let $r \geq 2$. By Proposition ~\ref{proposition:Legendre},
for each $j \not=j_*$, we have
\begin{align}
  \frac{v_{p_j}(n_0!)}{a_j}-\frac{v_q(n_0!)}{a}
  &=\frac{n_0}{\Theta_j}-\frac{n_0}{\Xi_q}
  -\frac{s_{p_j}(n_0)}{\Theta_j}+\frac{s_q(n_0)}{\Xi_q}\nonumber\\
  &=n_0\left(\frac{1}{\Theta_j}-\frac{1}{\Xi_q}\right)
  -\frac{s_{p_j}(n_0)}{\Theta_j}+\frac{s_q(n_0)}{\Xi_q}. \label{eq:sub2}
\end{align}
Using the bound (\ref{eq:s_bound}), that is, $s_{p_j}(n_0)\leq (p_j-1)(\log_{p_j}\!n_0+1)$, and
$s_q(n_0)=\mu(q-1)+s_q(u)$, (\ref{eq:sub2}) yields that
\begin{equation}\label{eq:sub1}
  \frac{v_{p_j}(n_0!)}{a_j}-\frac{v_q(n_0!)}{a}
  \geq n_0\left(\frac{1}{\Theta_j}-\frac{1}{\Xi_q}\right)
  -\frac{\log_{p_j}\! n_0+1}{a_j}
  +\frac{\mu}{a}+\frac{s_q(u)}{\Xi_q}.
\end{equation}

In the case where $\Theta_j<\Xi_q$, the first term of (\ref{eq:sub1}) increases exponentially in $\mu$. Hence, for all sufficiently large $i$, the whole right-hand side is positive.

Consider the case where $\Theta_j=\Xi_q$. The first term of (\ref{eq:sub1}) vanishes. Since $u$ satisfies $u \leq q^{\alpha\mu}$, we have
Then
\begin{align*}
\log_{p_j}\!n_0
\leq \log_{p_j}\! \left(q^\mu(1+u)\right)
&\leq \log_{p_j} \!\left\{q^{\mu(\alpha+1)}(q^{-\alpha\mu}+1)\right\}\\
&= \mu(\alpha+1)\frac{\log q}{\log p_j}+O(1) \qquad (\mu \rightarrow \infty).
\end{align*}
Thus, from (\ref{eq:sub1}), we have
\begin{equation}
\frac{v_{p_j}(n_0!)}{a_j}-\frac{v_q(n_0!)}{a}
\geq \mu\!\left(\frac{1}{a}-\frac{(1+\alpha)\log q}{a_j\log p_j}\right) + O(1). \label{eq:sub3}
\end{equation}
Hence, for sufficiently large $i$, the right-hand side of (\ref{eq:sub3}) is positive.

In both cases, we have $v_{p_j}(n_0!)/a_j > v_q(n_0!)/a$ for all $j \neq j_*$ as long as $\mu$ is sufficiently large.
By Lemma~\ref{lemma:vq-mod-a}, $v_q(n_0!)/a \in \NN$. Therefore we have the assertion.
\end{proof}

Put
\[
  d=\lnzd{b}{n_0!},\qquad
  t=v_b(n_0!).
\]
Write $n_0!=b^{t}U_b(n_0!)=q^{at}K^{t}U_b(n_0!)$. Then, by definition, $d \equiv U_b(n_0!) \pmod{b}$, that is, $d \equiv U_b(n_0!) \pmod{q^a}$. By Claim~\ref{claim}, we have $v_q(U_b(n_0!))=v_q(n_0!)-at=0<a$. By Lemma~\ref{lemma:conservation}, we have $v_q(d)=v_q(U_b(n_0!))=0$, that is, $\gcd(q,d)=1$.

\medskip
\noindent
\textit{Step 2.}
Let $A=n_0+\lambda$ and $B=n_0+\kappa \lambda$, where we recall that $\lambda$ is a period of $(\lnzd{b}{n!})_{n=1}^\infty$ and put
\[
  \kappa= \begin{cases}
    2&(\text{$q$ is odd}),\\
    1+2^{a-1}&(q=2).
  \end{cases}
\]
Then
\begin{align*}
A+1&=q^\mu(1+u)+q^{ai}K^i L,\\
B+1&=q^\mu(1+u)+q^{ai}K^i (\kappa L).
\end{align*}
Since $u\equiv -1\pmod{K^i}$, we have $K^i\mid (n_0+1)$. Thus $b^i\mid A+1,B+1$. Moreover
\[
A+1=q^{ai}\Big\{K^iL+q^{\mu-ai}(1+u)\Big\},\qquad \mu-ai\geq a \geq 1.
\]
Since $\gcd(K^iL,q)=1$, $v_q(A+1)=ai$. Therefore $v_b(A+1)=i$. Consequently, we have
\begin{align*}
\lnzd{b}{A+1}
\equiv \frac{A+1}{b^{i}}
&=\frac{q^{ai}\big\{q^{\mu-ai}(1+u)+K^iL\big\}}{q^{ai}K^i}\\
&\equiv q^{\mu-ai}(1+u)K^{-i}+L
\equiv L \pmod{q^a}.
\end{align*}
In the same way as above, we can show that $v_b(B+1)=i$ and $\lnzd{b}{B+1}\equiv \kappa L \pmod{q^a}$.

\medskip
\textit{Step 3.}
The periodicity \eqref{eq:periodicity} implies that $\lnzd{b}{A!}=\lnzd{b}{B!}=d$.
Since $\gcd(q,d)=1$ and applying Lemma~\ref{lemma:multiplicative} modulo $q^a$, we have
\[
\lnzd{b}{(A+1)!}\equiv dL,\qquad
\lnzd{b}{(B+1)!}\equiv d\,(\kappa L)\pmod{q^a}.
\]
Since $\gcd(L,q)=\gcd(d,q)=1$, we obtain
\begin{align*}
  &(dL)^{-1}\lnzd{b}{(A+1)!} \equiv 1\pmod{q^a},\\
  &(dL)^{-1}\lnzd{b}{(B+1)!} \equiv \kappa \pmod{q^a}.
\end{align*}
On the other hand, the periodicity forces $\lnzd{b}{(A+1)!}=\lnzd{b}{(B+1)!}$, that is, $\kappa-1 \equiv 0 \pmod{q^a}$. However $v_q(\kappa-1)<a$. This is a contradiction. We have proven the theorem.
\end{proof}

\begin{lemma}\label{lemma:vq-mod-a}
   Let $\mu$ and $u$ be chosen as in Step~1 in the proof of Theorem~\ref{theorem:MainResult2Sub}. Then we have $v_q(n_0!) \equiv 0 \pmod{a}$.
\end{lemma}

\begin{proof}
First we show that $v_q(n_0!) \equiv 0 \pmod{q^{v_{q}(a)}}$. By the construction of $n_0$, we have $n_0 \equiv -1 \pmod{q^\mu}$. Since $q^{v_q(a)} \leq a \leq \mu \leq q^\mu$, we obtain $n_0 \equiv -1 \pmod{q^{v_q(a)}}$. Moreover, Lemma~\ref{lemma:nocarry} and the choice of $u$ imply that $s_q(n_0)=\mu(q-1)+s_{q}(u) \equiv s_q(u) \equiv -1 \pmod{q^{v_q(a)}}$.\ Therefore we obtain $n_0-s_q(n_0) \equiv 0 \pmod{q^{v_{q}(a)}}$. Since $\gcd(q-1,q^{v_{q}(a)})=1$, the inverse $(q-1)^{-1}$ exists modulo $q^{v_{q}(a)}$. Hence, by Proposition~\ref{proposition:Legendre}, we have
\[
v_q(n_0!)=\frac{n_0-s_q(n_0)}{q-1}\equiv 0 \pmod{q^{v_{q}(a)}}.
\]

Next we show that $v_q(n_0!) \equiv 0 \pmod{c}$. Recall that $n_0=q^\mu(1+u)-1$. For $1\leq j\leq \mu$,
\[
\Big\lfloor \frac{n_0}{q^j}\Big\rfloor= q^{\mu-j}(1+u)-1.
\]
Therefore, by Proposition~\ref{proposition:Legendre}, we have
\[
v_q(n_0!)=\sum_{j=1}^{\mu}\left(q^{\mu-j}(1+u)-1\right)
=(1+u)\sum_{k=0}^{\mu-1} q^k-\mu=(1+u)S_\mu-\mu \equiv (1+u)S_\mu \pmod{c},
\]
where we put $S_\mu=\dsum_{k=0}^{\mu-1} q^k$. By the choice of $u$, we have $1+u \equiv 0 \pmod{c_2}$. Moreover, since $q \equiv 1 \pmod{c_1}$, $S_\mu \equiv \mu \equiv 0 \pmod{c_1}$. Namely, we have
\[
  v_q(n_0!) \equiv (1+u)S_\mu \equiv 0 \pmod{c}.
\]

Finally, since $\gcd(q^{v_q(a)},c)=1$, we have
\[
  v_q(n_0!) \equiv 0 \pmod{a}.
\]
\end{proof}

\subsection{An application}

Finally, as an application of our main results, we present examples of transcendental numbers. We begin by providing a simple sufficient condition for a given sequence to be a $\B$-expansion.

Let $\B=(\beta_k)_{k=1}^\infty$ be a real Cantor base of period $p$. Put $\delta=\prod_{i=1}^p\beta_i$ and $\beta_{\mathrm{min}}=\min_i \beta_i$. For a given sequence $\bm{a}=(a_n)_{n=1}^\infty$, we define $\mathrm{val}_\B(\bm{a})$ by the Cantor base representation
\[
  \mathrm{val}_\B(\bm{a})=\sum_{k=1}^\infty a_k\prod_{i=1}^k\beta_{i}^{-1}.
\]
We put $s(x)=x/(x-1)$ and
\[
D_\B=\{d_\B(x) \mid| 0 \leq x <1\}.
\]

\begin{proposition}\label{proposition:our_lemma1}
  Let $\B=(\beta_i)_{i=1}^\infty$ be a real Cantor base of period $p$. Put $\mathcal{A}=\{0,1,2,\cdots,[\max_i \beta_i]\}$. Let $\bm{a}=a_1a_2a_3\cdots$ be an infinite word on $\mathcal{A}$ which is not constant. If
  \[
    \max_k a_k
    \leq \frac{s(\beta_{\mathrm{min}}^p)}{s(\delta)}(\beta_{\mathrm{min}}-1),
  \]
  then $\bm{a} \in D_{\B}$.
\end{proposition}

\begin{proof}
  Let $\sigma$ denote the usual shift operator on $\mathcal{A}^\NN$. Let $n$ be a nonnegative integer. We define a sequence $\B^{(n)}$ by $\B^{(n)}=(\beta_n,\beta_{n+1},\cdots)$. Put $M_{\bm{a}}=\max_k a_k$. By the assumption, we have
  \begin{align*}
    \mathrm{val}_{\B^{(n)}}(\sigma^{n}(\bm{a}))
    =\sum_{k=1}^\infty a_{n+k}\prod_{i=1}^{k}\beta_{n+i}^{-1}
    &< M_{\bm{a}} \sum_{k=1}^\infty \prod_{i=1}^{k}\beta_{n+i}^{-1}\\
    &=M_{\bm{a}} \sum_{q=0}^\infty \sum_{r=1}^{p} \prod_{i=1}^{pq+r}\beta_{n+i}^{-1}\\
    &=M_{\bm{a}} \sum_{q=0}^\infty\delta^{-q} \sum_{r=1}^{p} \prod_{i=1}^{r}\beta_{n+i}^{-1}\\
    &\leq M_{\bm{a}} \sum_{q=0}^\infty\delta^{-q} \sum_{r=1}^{p} \beta_{\mathrm{min}}^{-r}
    = \frac{M_{\bm{a}}\delta (\beta_{\mathrm{min}}^{p}-1)}{\beta_{\mathrm{min}}^p(\delta-1)(\beta_{\mathrm{min}}-1)}
    \leq 1.
  \end{align*}
  Namely, we have $\mathrm{val}_{\B^{(n)}}(\sigma^{n}(\bm{a}))<1$ for all nonnegative integers $n$. Thus The claim follows from \cite[Proposition 4]{Charlier}.
\end{proof}

\begin{corollary}\label{corollary:lnzd-general}
Let $b>2$ be an integer with the prime factorization $b=p_1^{a_1}p_2^{a_2}\cdots$ such that $a_1(p_1-1) \geq a_2(p_2-1) \geq \cdots$. Let $\B=(\beta_k)_{k=1}^\infty$ be a real Cantor base of period $p$ consisting of real numbers $\beta_k>1$ such that $\delta=\prod_{k=1}^p\beta_k$ is Pisot. If $(b-1)s(\delta) \leq (\beta_{\mathrm{min}}-1)s(\beta_{\mathrm{min}}^p)$ and either $a_1(p_1-1)>a_2(p_2-1)$ or $b=p_1^{a_1}$, then the number
\begin{align*}
  \alpha_{b,\B} := \sum_{n=1}^\infty \lnzd{b}{n!}\prod_{k=1}^n \beta_k^{-1}
\end{align*}
is transcendental.
\end{corollary}

\begin{proof}
  According to Lipka \cite[Theorem 3.7]{Lipka}, the sequence $(\lnzd{b}{n!})_{n=1}^\infty$ is $p_1$-automatic if $a_1(p_1-1)>a_2(p_2-1)$ or $b=p_1^{a_1}$ and not automatic otherwise. Thus, by Proposition~\ref{proposition:stam_if_k-auto}, $(\lnzd{b}{n!})_{n=1}^\infty$ is $(\omega,M)$-stammering with $\omega>1$. Moreover, by Theorem~\ref{theorem:MainResult2} the sequence $(\lnzd{b}{n!})_{n=1}^\infty$ is not eventually periodic for every $b>2$. Therefore, combining Theorem~\ref{theorem:Schmidt-general} and Proposition~\ref{proposition:our_lemma1}, Corollary~\ref{corollary:Pisot-simplifies} yields that $\alpha_{b,\B}$ is transcendental.
\end{proof}

\bigskip

\begin{flushleft}\footnotesize
Kohta Gejima\\
	\textit{Otani University, Koyama Kamifusa-cho, Kita-ku, Kyoto 603-8143, Japan.}\\
	{\tt k.gejima@res.otani.ac.jp}
\end{flushleft}

\begin{flushleft}\footnotesize
Fumichika Takamizo\\
	\textit{Sanyo-Onoda City University, 1-1-1 Daigakudori, Sanyo-Onoda, Yamaguchi 756-0884, Japan.}\\
	{\tt f.takamizo@rs.socu.ac.jp}
\end{flushleft}
\end{document}